\documentclass[]{interact}
\usepackage{epstopdf,cite}
\usepackage[caption=false]{subfig}
\usepackage[numbers,sort&compress]{natbib}
\bibpunct[, ]{[}{]}{,}{n}{,}{,}
\makeatletter
\def\NAT@def@citea{\def\@citea{\NAT@separator}}
\makeatother

\theoremstyle{plain}
\newtheorem{theorem}{Theorem}[section]
\newtheorem{lemma}[theorem]{Lemma}
\newtheorem{corollary}[theorem]{Corollary}

\theoremstyle{definition}
\newtheorem{definition}[theorem]{Definition}

\theoremstyle{remark}
\newtheorem{remark}{Remark}

\usepackage{graphicx,color}

\newcommand{\R}{\mathbb R}

\def\dom{\mathop{\rm dom\,}}

\def\argmin{\mathop{\rm argmin\,}}

\usepackage{hyperref}
\hypersetup{
colorlinks=true,
linkcolor=red,
citecolor=red,
filecolor=magenta,
urlcolor=cyan
}
\usepackage{empheq}

\definecolor{myblue}{rgb}{.8, .8, 1}
\newcommand*\mybluebox[1]{
\colorbox{myblue}{\hspace{1em}#1\hspace{1em}}}

\usepackage{cleveref}
\crefname{chapter}{Chapter}{Chapters}
\crefname{item}{item}{items}
\crefname{figure}{Figure}{Figures}
\crefname{theorem}{Theorem}{Theorems}
\crefname{lemma}{Lemma}{Lemmas}
\crefname{proposition}{Proposition}{Propositions}
\crefname{corollary}{Corollary}{Corollarys}
\crefname{definition}{Definition}{Definitions}
\crefname{fact}{Fact}{Facts}
\crefname{example}{Example}{Examples}
\crefname{algorithm}{Algorithm}{Algorithms}
\crefname{remark}{Remark}{Remarks}
\crefname{note}{Note}{Notes}
\crefname{notation}{Notation}{Notations}
\crefname{case}{Case}{Cases}
\crefname{exercise}{Exercise}{Exercises}
\crefname{question}{Question}{Questions}
\crefname{claim}{Claim}{Claims}
\crefname{enumi}{}{}
\usepackage{empheq}

\begin{document}


\title{An inertial proximal splitting algorithm for hierarchical bilevel equilibria in Hilbert spaces}
\author{
\name{Aicha Balhag\textsuperscript{a}, Zakaria Mazgouri\textsuperscript{b}, Hassan Riahi\textsuperscript{c} and Michel Th\'era\textsuperscript{d}   }
\affil{\textsuperscript{a}Ecole Normale Sup\'erieure, Universit\'e Cadi Ayyad, Marrakesh, Morocco \\ 
	\textsuperscript{b}National School of Applied Sciences, Sidi Mohamed Ben Abdellah University, 30000 Fez, Morocco\\ORCID 0000-0003-0131-4741\\
	\textsuperscript{c}Universit\'e Cadi Ayyad, Facult\'e des Sciences Semlalia, D\'epartement de Math\'ematiques, \,B.P. 2390, Marrakesh, Morocco \\ 
	\textsuperscript{d}XLIM UMR-CNRS 7252, Universit\'e de Limoges, Limoges, France \\ORCID 0000-0001-9022-6406.
}
\vspace{5mm}
{}
\vspace{5mm}
\textit{We dedicate this article to  Professor Werner Oettli whose work strongly influenced the theory of variational inequalities, equilibrium problems,  and optimization.
}
}
\maketitle

\maketitle

\begin{abstract}
In this article, we aim to approximate a solution to the bilevel equilibrium problem $\mathbf{(BEP})$ for short: find \(\bar{x} \in \mathbf{S}_f\) such that  
$
g(\bar{x}, y) \geq 0, \,\, \forall y \in \mathbf{S}_f,
$
where  
$
\mathbf{S}_f = \{ u \in \mathbf{K} : f(u, z) \geq 0, \forall z \in \mathbf{K} \}.
$
Here, \(\mathbf{K}\) is a closed convex subset of a real Hilbert space \(\mathcal{H}\), and \( f \) and \( g \) are two real-valued bifunctions defined on \(\mathbf{K} \times \mathbf{K}\).  
We propose an inertial version of the proximal splitting algorithm introduced by Z. Chbani and H. Riahi:  
\textit{Weak and strong convergence of prox-penalization and splitting algorithms for bilevel equilibrium problems}.  
\textit{Numer. Algebra Control Optim.}, 3 (2013), pp. 353-366.  
Under suitable conditions, we establish the weak and strong convergence of the sequence generated by the proposed iterative method. {We also report a numerical example illustrating our theoretical result.}
\end{abstract}


 \begin{amscode} 90C33; 49J40; 46N10; 65K15; 65K10 \end{amscode}

\section{Introduction}
Throughout $\mathcal{H}$ is a real Hilbert space with identity operator $ \text{Id}$, scalar product $\langle \cdot,\cdot \rangle$  {and} associated norm $\Vert \cdot \Vert$. 
Let $\textbf{K}$ be a nonempty, closed, and convex subset {of $\mathcal{H}$}, and let $f, g:\textbf{K}\times \textbf{K} \rightarrow \R$ be two real-valued bifunctions. This work focuses on the following bilevel equilibrium problem:

Find $\bar{x} \in \textbf{S}_f$ such that
\begin{empheq}[box =\mybluebox]{equation*}\label{BEP}
\tag{$\mathbf{BEP}$ }%
g(\bar{x},y)\geq 0, \;\; \forall y\in \textbf{S}_f,
\end{empheq}
where the constraint set $\textbf{S}_f$ represents the solution set of the following equilibrium problem associated with $f$:

Find $y\in \textbf{K}$ such that
\begin{empheq}[box =\mybluebox]{equation*}\label{EP}
\tag{$\mathbf{EP}$}%
f(y,u)\geq 0, \;\; \forall u\in \textbf{K}.
\end{empheq}
In other words, problem $\mathbf{(BEP)}$ is an equilibrium problem where the constraint set consists of solutions to another equilibrium problem. The formulation of $\mathbf{(BEP)}$ is widely recognized as a general framework that unifies many classical problems in nonlinear analysis and optimization. These include hierarchical minimization problems, mathematical programming, optimization problems with equilibrium or fixed point constraints, optimal control problems governed by state equations defined through variational or quasi-variational inequalities, Nash equilibrium concepts, and more. For further details, see \cite{Dempe} and references therein.

Throughout this work, we assume that the solution set of $\mathbf{(BEP)}$, denoted by $\mathbf{S}$, is nonempty.
\vskip 2mm 
Although bilevel equilibrium problems are interesting, they are challenging to solve due to the presence of a bifunction at both levels.

The classical approach for approximating a solution to $\mathbf{(BEP)}$ is the \textit{regularized proximal point algorithm}  $\textbf{(RPPM)}$, introduced by Moudafi \cite{moud2}. This algorithm generates a sequence $\{x_n\}$ according to the iteration:
\[
x_{n+1}=J_{\lambda_n}^{f+\beta_ng} (x_n),
\]
which satisfies
\begin{empheq}[box =\mybluebox]{equation}\label{PPM-BEP}
f(x_{n+1},y)+\beta_n g(x_{n+1},y)+\dfrac{1}{\lambda_n}\langle x_{n+1}-x_n, y-x_{n+1} \rangle \geq 0, \;\; \forall y\in \textbf{K},
\end{empheq}
where $\{\beta_n\}$ and $\{\lambda_n\}$ are two sequences of nonnegative real numbers, and $J_{\lambda_n}^{f+\beta_ng}$ is the resolvent associated with the parametrized family of bifunctions $f+\beta_ng$.

It turns out that the sequence generated by this algorithm converges weakly to a solution of $\mathbf{(BEP)}$ under the conditions:
\[
\liminf_{n\rightarrow +\infty} \lambda_n>0, \quad \sum_{n=0}^{+\infty} \lambda_n \beta_n <+\infty, \quad \text{and} \quad \|x_{n+1}-x_n\|=o(\beta_n).
\]
However, a major challenge in applying this method lies in the last assumption, as it requires choosing a suitable control sequence $(\beta_n)$ without knowing the convergence rate of $\|x_{n+1} - x_n\|$.

To address this issue, an alternative proximal scheme was proposed in \cite{CR1}, which generates the sequence using the iteration:
\[
x_{n+1}:=J_{\lambda_n}^{\beta_nf+g} (x_n),
\]
i.e.,
\begin{empheq}[box =\mybluebox]{equation}\label{PPM-naco}
\beta_n f(x_{n+1},y)+g(x_{n+1},y)+\dfrac{1}{\lambda_n}\langle x_{n+1}-x_n, y-x_{n+1} \rangle \geq 0, \;\; \forall y\in \textbf{K}.
\end{empheq}

This modification overcomes the difficulty highlighted in \cite{moud2} by introducing a geometric assumption based on the Fenchel conjugate function associated with the bifunction $f$. Consequently, the authors  in \cite{CR1} analyzed both weak and strong convergence of the algorithm to a solution of $\mathbf{(BEP)}$.\\

In recent years, there has been significant interest in incorporating inertial terms into various proximal-type algorithms. Notably, the inclusion of inertial terms enhances the numerical performance of these algorithms by accelerating their convergence. This acceleration occurs because each new iterate is computed as a combination of the previous two iterates, which can also be interpreted as an implicit discretization of second-order dynamical systems in time.

The origins of these methods can be traced back to the work of Alvarez and Attouch \cite{AA}. Inspired by Polyak's heavy ball method \cite{polyak1964some}, originally introduced in the context of minimizing smooth convex functions, the authors in \cite{AA} developed the inertial proximal point algorithm for finding zeros of general maximally monotone operators.

A more recent investigation \cite{balhag2022weak} extended the proximal algorithm \eqref{PPM-naco} by incorporating inertial terms. The authors in \cite{balhag2022weak} proposed the following inertial proximal scheme for solving $\mathbf{(BEP)}$:
\[
x_{n+1}:=J_{\lambda_n}^{\beta_nf+g} (y_n),
\]
which satisfies
\begin{empheq}[box =\mybluebox]{equation}\label{algo-siopt}
\beta_n f(x_{n+1},y)+g(x_{n+1},y)+\frac{1}{\lambda_n}\langle x_{n+1}-y_n, y-x_{n+1}\rangle \geq 0, \;\;\forall y\in \textbf{K},
\end{empheq}
where the inertial term is defined as $y_n:=x_n+\alpha_n(x_n-x_{n-1})$, with $\{\alpha_n\}$ being a nonnegative sequence satisfying $\{\alpha_n\} \subseteq [0, 1]$.

It was shown that under suitable conditions on the positive parameters $\beta_n$ and $\lambda_n$, and by assuming that the sequence $\{\alpha_n\}$ is nondecreasing and satisfies $\{\alpha_n\} \subseteq [0,\alpha]$ for some $\alpha\geq 0$ such that $0\leq \alpha< \frac{1}{3}$, the algorithm converges weakly to a solution of $\mathbf{(BEP)}$. This result was established using a discrete counterpart of the geometric condition introduced in \cite{CMR}. Moreover, the authors also proved strong convergence under a geometric hypothesis and even in the absence of such {assumption}.\\

Another important class of iterative methods in optimization is given by the so-called splitting methods. Since evaluating the resolvent operator of the sum of two maximally monotone operators (or monotone bifunctions) is generally difficult, an alternative approach relies on splitting techniques. These methods are based on the idea of separately computing the resolvents of each operator (or bifunction), which is often easier to evaluate.

In the context of monotone inclusions, various splitting-type algorithms have been extensively studied. For a detailed overview, the interested reader may refer to \cite{lorenz2015inertial, moudafi2003convergence, passty1979ergodic} and the references therein.

To approximate a solution of
\begin{equation}\label{Anc}
0\in Ax+N_C(x),
\end{equation}
where $A: \mathcal{H} \rightrightarrows \mathcal{H}$ is a maximally monotone operator and $N_C$ is the outward normal cone to a closed convex set $C\subset \mathcal{H}$, Attouch et al. \cite{attouch2011prox} proposed the following splitting algorithm:
\begin{equation}
\left\{\begin{array}{l}
y_n=(I+\lambda_nA)^{-1}x_{n-1}, \\
x_n=(I+\lambda_n\beta_n\partial \psi)^{-1}y_n,
\end{array}\right.
\end{equation}
where $\{\beta_n\}$ and $\{\lambda_n\}$ are sequences of nonnegative real numbers, and $\psi: \mathcal{H} \rightarrow \R\cup \{+\infty\}$ acts as a penalization function enforcing the constraint $x\in C$. The authors established conditions ensuring either the ergodic or strong convergence of the sequence generated by their algorithm to a solution of \eqref{Anc}.

Motivated by these works, several studies have extended splitting methods to the bilevel equilibrium context. The algorithm proposed in \cite{passty1979ergodic} was later generalized by Moudafi \cite{moudafi2009convergence} for solving the problem
\[
f(x,y)+g(x,y)\geq 0, \quad \forall y\in \textbf{K}.
\]
Chbani and Riahi \cite{CR1} introduced a similar splitting algorithm:
\begin{equation}\label{psm}
\left\{\begin{array}{l}
g(z_{n+1},y)+\frac{1}{\lambda_n}\langle z_{n+1}-x_n, y-z_{n+1}\rangle \geq 0, \quad \forall y\in \textbf{K}, \\
\beta_n f(x_{n+1},y)+\frac{1}{\lambda_n}\langle x_{n+1}-z_{n+1}, y-x_{n+1}\rangle \geq 0, \quad \forall y\in \textbf{K},
\end{array}\right.
\end{equation}
which includes a viscosity term $\beta_n$ to ensure convergence to a solution of $\mathbf{(BEP)}$.

\medskip
In this work, we focus on proximal splitting methods for solving $\mathbf{(BEP)}$ and aim to incorporate inertial terms into the algorithm \eqref{psm}. Our proposed method combines the inertial iterative scheme \eqref{algo-siopt} with the splitting algorithm \eqref{psm}. Inspired by the approach presented in \cite{balhag2022weak}, we introduce the following numerical scheme: 
\smallskip
\hrule
\vspace{2mm}
\noindent {\bf Algorithm:} \; Inertial Proximal Splitting Algorithm $\textbf{(IPSA)}$
\vspace{2mm}
\hrule
\vspace{2mm}
\noindent {\bf Initialization:} Choose positive sequences $\{\beta_n\}$, $\{\lambda_n\}$, and a nonnegative sequence $\{\alpha_n\} \subseteq [0,1]$. Select arbitrary initial points $x_0, x_1 \in \textbf{K}$.
\vspace{2mm}
\hrule
\vspace{2mm}
\noindent {\bf Iterative Step:} For every $n\geq 1$, given the current iterates $x_{n-1}, x_n \in \textbf{K}$, set
\[
y_n:=x_n+\alpha_n(x_n-x_{n-1}),
\]
and define $x_{n+1}\in \textbf{K}$ by
\[
x_{n+1}:=J_{\lambda_n}^{\beta_nf} \circ J_{\lambda_n}^{g} (y_n),
\]
i.e.,
\begin{equation}\label{algo}
\left\{\begin{array}{l}
g(z_{n+1},y)+\frac{1}{\lambda_n}\langle z_{n+1}-y_n, y-z_{n+1}\rangle \geq 0, \quad \forall y\in \textbf{K}, \\
\beta_n f(x_{n+1},y)+\frac{1}{\lambda_n}\langle x_{n+1}-z_{n+1}, y-x_{n+1}\rangle \geq 0, \quad \forall y\in \textbf{K}.
\end{array}\right.
\end{equation}
\hrule
\vspace{5mm}
In this algorithm, $\{\lambda_n\}$ represents the sequence of step sizes, $\{\beta_n\}$ the sequence of penalization parameters, and $\{\alpha_n\}$ the sequence of inertial parameters. If $\alpha_n = 0$ for all $n\geq 0$, then $\textbf{(IPSA)}$ reduces to the original splitting algorithm \eqref{psm}.\\

In what follows, we outline the structure of the remainder of the paper.

In Section 2, we present the necessary background for our analysis. Section 3 is dedicated to the convergence analysis of algorithm $\textbf{(IPSA)}$. We establish conditions under which the sequence generated by this algorithm converges weakly or strongly to a solution of $\mathbf{(BEP})$.

More specifically, in Subsection 3.1, using a discrete counterpart \eqref{fitz-discret} of the geometric condition introduced in \cite{CMR} -- formulated in terms of the Fitzpatrick transform of the bifunction $f$ -- we prove (see Theorem \ref{thm-weak}) that the sequence generated by $\textbf{(IPSA)}$ weakly converges to a solution of $\mathbf{(BEP})$, provided that:
\[
{\left\{\lambda_n \right\} \in \ell^2 \setminus \ell^1, \quad \liminf_{n\rightarrow +\infty}\lambda_n \beta_n>0},
\]
and that the sequence $\{\alpha_n\}$ is nondecreasing and satisfies $\{\alpha_n\} \subseteq [0,\alpha]$ for some $\alpha \geq 0$ with $\alpha< \frac{\sqrt{3}-1}{4}$.

In Subsection 3.2, by strengthening the monotonicity assumption on the upper-level bifunction $g$ and without requiring the geometric assumption \eqref{fitz-discret}, we show (see Theorem \ref{strong1}) that the sequence generated by $\textbf{(IPSA)}$ strongly converges to the unique solution of $\mathbf{(BEP})$ under the conditions:
\[
\{\alpha_n\} \text{ is nondecreasing, } \,\, \{\alpha_n\} \subseteq [0,\alpha] \text{ with } 0\leq \alpha< \frac{\sqrt{3}-1}{4},
\]
\[
\displaystyle\sum_{n\geq0} \lambda_n = +\infty, \, \lim_{n\rightarrow +\infty}\lambda_{n}=0, \, {\lim_{n\rightarrow +\infty} \beta_n = +\infty} \quad and \quad \liminf_{n\rightarrow +\infty}\lambda_n\beta_n>0.
\]

A key advantage of our approach is that it ensures convergence without imposing restrictive assumptions on the trajectories. To the best of our knowledge, such a proximal inertial {splitting} scheme has not yet been studied for the two-level problem $\mathbf{(BEP})$.

In Section 4, we present a numerical experiment to illustrate our theoretical results. The paper concludes with some technical auxiliary results provided in the Appendix.

\section{Notations and background}
For the reader's convenience, we first recall some well-known concepts related to the monotonicity and continuity of real-valued bifunctions.

\begin{definition}\label{d1}
A bifunction $f:\textbf{K}\times \textbf{K} \rightarrow \mathbb{R}$ is said to be:
\begin{itemize}
    \item [$(i)$] \textit{Monotone} if
    \[
    f(x,y)+f(y,x)\leq0, \quad \forall x,y \in \textbf{K}.
    \]
    \item [$(ii)$] \textit{$\gamma$-strongly monotone} if there exists $\gamma >0$ such that
    \[
    f(x,y)+f(y,x)\leq -\gamma \|x-y\|^2, \quad \forall x,y \in \textbf{K}.
    \]
    \item [$(iii)$] \textit{Upper hemicontinuous} if
    \[
    \lim_{t\searrow0} f(tz+(1-t)x,y)\leq f(x,y), \quad \forall x,y,z\in \textbf{K}.
    \]
    \item [$(iv)$] \textit{Lower semicontinuous at $y$ with respect to the second argument} if
    \[
    f(x,y)\leq \liminf_{w\rightarrow y}f(x,w), \quad \forall x\in \textbf{K}.
    \]
    \item [$(v)$] \textit{An equilibrium bifunction} if, for each $x\in \textbf{K}$, we have $f(x,x)=0$ and $ f(x,\cdot)$ is convex and lower semicontinuous.
\end{itemize}
\end{definition}

The dual equilibrium problem associated with the bifunction $f$ on $\textbf{K}$ is stated as follows:

Find $\bar{x}\in \textbf{K}$ such that
\begin{equation*}\label{dep}
\tag{\textbf{DEP}}
f(y,\bar{x})\leq 0, \quad \forall y\in \textbf{K}.
\end{equation*}
The set of solutions to $\textbf{(DEP)}$ is called the \textit{Minty solution set}. The following result establishes the connection between Minty equilibria and standard equilibria.

\begin{lemma}[Minty's Lemma, \cite{BO}]\label{lem-Minty}
\begin{itemize}
    \item[(i)] If $f$ is monotone, then every solution of $\textbf{(EP)}$ is also a solution of $\textbf{(DEP)}$.
    \item[(ii)] Conversely, if $f$ is upper hemicontinuous and an equilibrium bifunction, then every solution of $\textbf{(DEP)}$ is a solution of $\textbf{(EP)}$.
\end{itemize}
\end{lemma}
The following lemma introduces the notion of the \textit{resolvent} associated with bifunctions, a key concept in our approach to solving $\mathbf{(BEP)}$.

\begin{lemma}\cite{CR2} \label{lem0}
Let $f:\textbf{K}\times \textbf{K}\rightarrow \mathbb{R}$ be a monotone equilibrium bifunction. Then, the following statements are equivalent:
\begin{itemize}
    \item [(i)] $f$ is maximal: If $(x,u)\in \textbf{K}\times \mathcal{H}$ satisfies
    \[
    f(x,y)\leq \langle u,x-y \rangle, \quad \forall y\in \textbf{K},
    \]
    then it follows that
    \[
    f(x,y)+\langle u,x-y \rangle \geq 0, \quad \forall y\in \textbf{K}.
    \]
    \item [(ii)] For each $x \in \mathcal{H}$ and $\lambda > 0$, there exists a unique $z_{\lambda}=J^{f}_{\lambda}(x)\in \textbf{K}$, called the resolvent of $f$ at $x$, such that
    \begin{equation}\label{aa}
    \lambda f(z_{\lambda},y) +\langle y-z_{\lambda} ,z_{\lambda}-x  \rangle \geq 0, \quad \forall y\in \textbf{K}.
    \end{equation}
\end{itemize}
Moreover, $\bar{x}\in \textbf{S}_f$ if and only if $\bar{x} = J^{f}_{\lambda}(\bar{x})$ for every $\lambda>0$, which is also equivalent to $\bar{x} = J^{f}_{\lambda}(\bar{x})$ for some $\lambda>0$.
\end{lemma}

\begin{remark}
Assertion (ii) above, as stated in \cite[Lemma 2.1]{CCR}, holds under the assumptions that $f(x,x) = 0$ for each $x \in \textbf{K}$ and that $f$ is upper hemicontinuous and convex in its second argument.
\end{remark}


\subsection{Preliminaries from Convex Analysis}

In this section, we recall some fundamental concepts from convex analysis that will be used in our study.

For a function $\varphi: \mathcal{H}\rightarrow \mathbb{R}\cup\{+\infty\}$, we denote its \textit{effective domain} by
\[
\dom \varphi=\{x \in \mathcal{H}: \varphi(x)<+\infty\}.
\]
The function $\varphi$ is said to be \textit{proper} if $\dom \varphi\neq \emptyset$.

The \textit{optimal objective value} of $\varphi$ is defined as
\[
\min \varphi := \inf_{x\in \mathcal{H}} \varphi(x),
\]
and its \textit{set of global minima} is given by
\[
\argmin \varphi := \{x\in \mathcal{H}: \varphi(x)=\min \varphi\}.
\]

For a proper, lower semicontinuous, convex function $\varphi: \mathcal{H}\rightarrow \mathbb{R}\cup\{+\infty\}$, its \textit{Fenchel conjugate function} $\varphi^*: \mathcal{H} \rightarrow \mathbb{R}\cup\{+\infty\}$ is defined by
\[
\varphi^*(x):=\sup_{y \in \mathcal{H}}\{\langle x,y \rangle-\varphi(y)\}.
\]

If $\varphi$ is the \textit{indicator function} of a set $\textbf{K} \subset \mathcal{H}$, given by
\[
\delta_\textbf{K}(x) =
\begin{cases}
0, & \text{if } x\in \textbf{K}, \\
+\infty, & \text{otherwise},
\end{cases}
\]
then its Fenchel conjugate function at $x^* \in \mathcal{H}$ is the \textit{support function} of $\textbf{K}$, given by
\[
\delta_\textbf{K}^*(x^*)=\sigma_\textbf{K}(x^*)=\sup_{y\in \textbf{K}}\langle x^*,y \rangle.
\]

The \textit{subdifferential} of $\varphi$ at $x \in \mathcal{H}$, where $\varphi(x)\in \mathbb{R}$, is the set
\[
\partial \varphi(x):=\{v \in \mathcal{H}: \varphi(y)\geq \varphi(x)+\langle v,y-x\rangle, \quad \forall y \in \mathcal{H} \}.
\]
By convention, we set $\partial \varphi(x) = \emptyset$ if $\varphi(x)=+\infty$.
The \textit{normal cone} to $\textbf{K}\subset  \mathcal{H} $ at $x \in \mathcal{H}$ is the set 
$$N_\textbf{K}(x)=\left\{\begin{array}{ll}
\{x^* \in \mathcal{H}: \langle x^*,u-x\rangle\leq0,\;\forall u\in \textbf{K}\}&  \text{ if }  x\in \textbf{K}\\
\emptyset& \text{otherwise}.
\end{array}\right.$$
We mention that $N_\textbf{K}=\partial\delta_\textbf{K}$, and that $x^*\in N_\textbf{K}(x)$ if, and only if, $\sigma_\textbf{K}(x^*)=\langle x^*,x\rangle$.
For every $u\in \textbf{K}$, we denote by $f_u$ the function defined  on $\mathcal{H}$ by $f_u(x)=f(u,x)$ if $x\in \textbf{K}$ and $f_u(x)=+\infty$ otherwise.
For an equilibrium bifunction $f:\textbf{K}\times \textbf{K} \rightarrow \mathbb{R}$, the \textit{associated operator $A^f$} is defined by $$A^f(x):= \partial f_x(x)=\left\{\begin{array}{ll} \lbrace z \in \mathcal{H}:f(x,y)+\langle z,x-y \rangle\geq0, \;\forall y\in \textbf{K} \rbrace&  \text{ if }  x\in \textbf{K}\\
\emptyset& \text{otherwise}.
\end{array}\right.$$ 
Following \cite{AH,BG}, the \textit{Fitzpatrick transform}  $\mathcal{F}_f : \textbf{K}\times  \mathcal{H}\rightarrow \R\cup \lbrace +\infty \rbrace$ associated with a bifunction $f$ is defined by 
$$\mathcal{F}_f(x,u)=\displaystyle \sup_{y\in \textbf{K}}\lbrace \langle u, y\rangle+f(y, x)\rbrace.$$ Given its continuity and convexity properties, the function $\mathcal{F}_f $ has proven to be an important tool when studying the asymptotic properties of dynamical equilibrium systems,
see \cite{CMR} for a detailed presentation of these elements.

\subsection{Assumptions}

In the remainder of the paper, we assume the following conditions:

\begin{itemize}
    \item[$\blacktriangleright$] The bifunctions $f$ and $g$ are monotone and upper hemicontinuous, satisfying conditions $(iv)$ and $(v)$ of Definition \ref{d1}.
   
    \item[$\blacktriangleright$] For each $y\in \textbf{K}$, we have $\partial g_y(y)\neq \emptyset$ (i.e., $\dom (A^{g})=\textbf{K}$), and $\textbf{K} \cap \textbf{S}_f \neq \emptyset$.
   
    \item[$\blacktriangleright$] The set $\R_+ (\textbf{K}-\textbf{S}_f)$ is a closed linear subspace of $\mathcal{H}$. In this case, the operator $g_x+\delta_{\textbf{S}_f}$ is maximally monotone (see \cite{Att-Riahi-Thera, Riahi5}), and the subdifferential sum formula
    \[
    \partial (g_x+\delta_{\textbf{S}_f})=\partial g_x+N_{\textbf{S}_f}
    \]
    holds.
   
    \item[$\blacktriangleright$] The following geometric assumption holds: For all $u\in \textbf{S}_f$ and for every $p\in N_{\textbf{S}_f}(u)$,
    \begin{equation}\label{fitz-discret}
    \sum_{n=1}^{+\infty} \lambda_n\beta_n\left[\mathcal{F}_f\left(u,\frac{2p}{\beta_n}\right)-\sigma_{\textbf{S}_f}\left(\frac{2p}{\beta_n}\right)\right]<+\infty.
    \end{equation}
    Assumption \eqref{fitz-discret} serves as a key tool in our weak convergence analysis. Notably, this condition represents the discrete counterpart of the assumption introduced in \cite{CMR} in the context of continuous-time dynamical equilibrium systems. Furthermore, it extends similar assumptions commonly used in the convergence analysis of variational inequalities formulated as monotone inclusion problems and in constrained convex optimization problems. For further discussion on these assumptions, see \cite{attouch2011prox, BC1, BC2} and the references therein.
\end{itemize}

\section{The main results}\label{S2}
\subsection{Weak convergence analysis}\label{Sub2}

In this subsection, under natural conditions, we formulate and prove a weak convergence result for the trajectory generated by \eqref{algo} to a solution of $\mathbf{(BEP})$. We begin with the following preliminary estimate.

\begin{lemma}\label{lem5}
	Let $\lbrace x_n \rbrace$ be a sequence generated by algorithm \eqref{algo}. Take $u \in \textbf{S}$ and set $a_n:=\|x_n-u\|^2$. Then, there exist $p\in N_{\textbf{S}_f}(u)$ and $b>0$ such that for each $n\geq 1$, the following inequality holds:
	\begin{equation}\label{estim1}
	\begin{array}{l}
	a_{n+1}-a_n-\alpha_n(a_n-a_{n-1})+\lambda_n\beta_nf(u,x_{n+1})-\alpha_n\left(b+1\right)\|x_n-x_{n-1}\|^2\\
	\leq2\lambda_n^2\| p\|^2-(\frac{1}{2}-\frac{\alpha_n}{b})\|z_{n+1}-x_n\|^2-\frac{1}{4}\|x_{n+1}-x_{n}\|^2  +\lambda_n\beta_n\left[\mathcal{F}_f\left(u,\frac{2p}{\beta_n}\right)-\sigma_{\textbf{S}_f}\left(\frac{2p}{\beta_n}\right)\right].
	\end{array}
	\end{equation}
\end{lemma}

\begin{proof}
	Since $u\in \mathbf{S}$, the first-order optimality condition gives
	\[
	0\in \partial(g_u+\delta_{\textbf{S}_f})(u)=A^g(u)+N_{\textbf{S}_f}(u).
	\]
	Let $p\in N_{\textbf{S}_f}(u)$ be such that $-p\in A^g(u)$. Then, for every $n\geq 1$, we have
	\begin{equation}\label{a}
	\lambda_n g(u,y) + \lambda_n \langle -p,u-y\rangle \geq 0, \quad \forall y\in \textbf{K}.
	\end{equation} 
	From inequality \eqref{algo}, it follows that
	\begin{equation} \label{b}
	g(z_{n+1},y) + \frac{1}{\lambda_n} \langle z_{n+1}-y_n, y-z_{n+1}\rangle \geq 0, \quad \forall y\in \textbf{K}.
	\end{equation} 
	Setting \( y = z_{n+1} \) in \eqref{a} and \( y = u \) in \eqref{b}, and using the monotonicity of \( g \), we obtain
	\begin{equation*} 
	\lambda_n\langle -p,u-z_{n+1}\rangle+\langle z_{n+1}-y_n, u-z_{n+1}\rangle \geq 0,
	\end{equation*} 
	which is equivalent to
	\begin{equation} \label{c}
	\lambda_n\langle -p,u-z_{n+1}\rangle +\langle z_{n+1}-y_n, u-x_{n+1}\rangle +\langle z_{n+1}-y_n, x_{n+1}-z_{n+1}\rangle\geq 0.
	\end{equation} 
	Since the sequence \( \{ x_n\} \) is generated by \eqref{algo}, we also have
	\begin{equation}\label{d}
	\lambda_n \beta_n f(x_{n+1},u)+\langle x_{n+1}-z_{n+1}, u-x_{n+1}\rangle \geq 0.
	\end{equation}
	Summing these inequalities gives
	\begin{equation}\label{e}
	\lambda_n \beta_n f(x_{n+1},u)+\langle x_{n+1}-y_n, u-x_{n+1}\rangle+ \lambda_n\langle -p,u-z_{n+1}\rangle +\langle z_{n+1}-y_n, x_{n+1}-z_{n+1}\rangle \geq 0. 
	\end{equation}

	Using the identities
	\[
	\langle x_{n+1}-y_n, u-x_{n+1}\rangle= \frac{1}{2} \left[ \|y_n-u\|^2 - \|x_{n+1}-y_n\|^2 - \|x_{n+1}-u\|^2 \right],
	\]
	and
	\[
	\langle z_{n+1}-y_n, x_{n+1}-z_{n+1} \rangle = \frac{1}{2} \left[ \|x_{n+1}-y_n\|^2 - \|z_{n+1}-y_n\|^2 - \|x_{n+1}-z_{n+1}\|^2 \right],
	\]
	inequality \eqref{e} simplifies to
	\begin{equation}\label{f}
	\begin{split}
	&2 \lambda_n \beta_n f(x_{n+1},u) + \|y_n - u\|^2 - \|x_{n+1} - u\|^2 + 2\lambda_n\langle -p,u-z_{n+1}\rangle \\
	&\quad - \|x_{n+1}-z_{n+1}\|^2 - \|z_{n+1}-y_n\|^2 \geq 0.
	\end{split}
	\end{equation}

	We also have
	\begin{equation*}
	2\lambda_n\langle -p,u-z_{n+1}\rangle \leq 2\lambda_n\langle -p,u-x_{n+1}\rangle +2\lambda_n^2\| p\|^2+\frac{1}{2}\| x_{n+1}-z_{n+1}\|^2.
	\end{equation*}
	Substituting this into \eqref{f} gives
	\begin{equation}\label{f1}
	\begin{split}
	&2 \lambda_n \beta_n f(x_{n+1},u) + \|y_n - u\|^2 - \|x_{n+1} - u\|^2 + 2\lambda_n\langle -p,u-x_{n+1}\rangle \\
	&\quad -\frac{1}{2}\|x_{n+1}-z_{n+1}\|^2 - \|z_{n+1}-y_n\|^2 + 2\lambda_n^2\| p\|^2  \geq 0.
	\end{split}
	\end{equation}

	By Lemma \ref{lem5-a}, for all \( n\geq 1 \), we have
	\begin{align}\label{8a}
	\|y_n-u\|^2 = (1+\alpha_n)\|x_n-u\|^2-\alpha_n \|x_{n-1}-u\|^2+\alpha_n(1+\alpha_n)\|x_n-x_{n-1}\|^2.
	\end{align}

	Similarly, for some \( b>0 \), we get
	\begin{equation}\label{8b}
	\begin{split}
	\|z_{n+1}-y_n\|^2 &\geq (1-\frac{\alpha_n}{b} )\|z_{n+1}-x_n\|^2 + (\alpha_n^2-\alpha_n b)\|x_n-x_{n-1}\|^2.
	\end{split}
	\end{equation}	

	Combining \eqref{8a} and \eqref{8b} with \eqref{f1}, we obtain
	\begin{equation}\label{15principal}
	\begin{split}
	&\|x_{n+1}-u\|^2-(1+\alpha_n)\|x_n-u\|^2+\alpha_n\|x_{n-1}-u\|^2 \\ 
	&\leq  2 \lambda_n \beta_n f(x_{n+1},u)+\alpha_n(b+1)\|x_n-x_{n-1}\|^2 - (1-\frac{\alpha_n}{b})\|z_{n+1}-x_n\|^2 \\
	&\quad -\frac{1}{2}\|z_{n+1}-x_{n+1}\|^2 + 2\lambda_n\langle -p,u-x_{n+1}\rangle + 2\lambda_n^2\| p\|^2. 
	\end{split}
	\end{equation}
	
			Taking $a_n=\|x_n-u\|^{2}$ in \eqref{15principal}, and using the monotonicity of $f,$ yields 
		\begin{equation}\label{15c}
		\begin{split}
		&a_{n+1}-(1+\alpha_n)a_n+\alpha_n a_{n-1} \\
		&\leq -\lambda_n \beta_n f(u,x_{n+1})+\lambda_n \beta_n f(x_{n+1},u)+ 2\lambda_n\langle -p,u-x_{n+1}\rangle \\
		&\quad+\alpha_n(b+1)\|x_{n}-x_{n-1}\|^2-(1-\frac{\alpha_n}{b})\|z_{n+1}-x_n\|^2-\frac{1}{2}\|z_{n+1}-x_{n+1}\|^2 +2\lambda_n^2\|p\|^2 \\
		&= -\lambda_n	\beta_n f(u,x_{n+1})+\lambda_n	\beta_n \left[ \langle \frac{2p}{\beta_n},x_{n+1}\rangle+f(x_{n+1},u)-\langle \frac{2p}{\beta_n},u\rangle\right] \\
		&\quad+\alpha_n(b+1)\|x_{n}-x_{n-1}\|^2-(1-\frac{\alpha_n}{b})\|z_{n+1}-x_n\|^2-\frac{1}{2}\|z_{n+1}-x_{n+1}\|^2 +2\lambda_n^2\| p\|^2 .
		\end{split}
		\end{equation}
		Since $p \in N_{S_f} (u),$ we have $\sigma_{S_f} (\frac{2p}{\beta_n})=\langle\frac{2p}{\beta_n},u\rangle.$ Hence
		\begin{equation}\label{15c2}
		\begin{split}
		&a_{n+1}-(1+\alpha_n)a_n+\alpha_n a_{n-1}+\lambda_n	\beta_n f(u,x_{n+1}) \\
		&\leq \lambda_n\beta_n \left[\langle \frac{2p}{\beta_n},x_{n+1}\rangle+f(x_{n+1},u)-\sigma_{S_f} (\frac{2p}{\beta_n})\right] +\alpha_n(b+1)\|x_{n}-x_{n-1}\|^2
		\\
		&\quad-(1-\frac{\alpha_n}{b})\|z_{n+1}-x_n\|^2-\frac{1}{2}\|z_{n+1}-x_{n+1}\|^2 +2\lambda_n^2\| p\|^2 \\
		&\leq\lambda_n\beta_n \left[\sup_{x \in \mathcal{H}} \{\langle \frac{2p}{\beta_n},x \rangle+f(x,u)\}-\sigma_{S_f}(\frac{2p}{\beta_n}) \right] 
		+\alpha_n(b+1)\|x_{n}-x_{n-1}\|^2\\
		&\quad-(1-\frac{\alpha_n}{b})\|z_{n+1}-x_n\|^2 -\frac{1}{2}\|z_{n+1}-x_{n+1}\|^2 +2\lambda_n^2\|p\|^2 \\
		&=\lambda_n\beta_n\left[\mathcal{F}_f\left(u,\frac{2p}{\beta_n}\right)-\sigma_{S_f}\left(\frac{2p}{\beta_n}\right)\right] +\alpha_n(b+1)\|x_{n}-x_{n-1}\|^2
		-(1-\frac{\alpha_n}{b})\|z_{n+1}-x_n\|^2 \\
		&\quad-\frac{1}{2}\|z_{n+1}-x_{n+1}\|^2 +2\lambda_n^2\| p\|^2.
		\end{split}
		\end{equation}		
		We also have 
	\begin{equation}\label{zn}
	\begin{split}
	&-\frac{1}{2}\|z_{n+1}-x_{n+1}\|^2  = -\frac{1}{2}\|z_{n+1}-x_{n}\|^2 -\frac{1}{2}\|x_{n}-x_{n+1}\|^2 +\langle z_{n+1}-x_{n},x_{n+1}-x_{n}\rangle \\
	& \leq  -\frac{1}{2}\|z_{n+1}-x_{n}\|^2 -\frac{1}{2}\|x_{n+1}-x_{n}\|^2 +\left(\|z_{n+1}-x_{n}\|^2+\frac{1}{4}\|x_{n+1}-x_{n}\|^2  \right)\\
	& = \frac{1}{2}\|z_{n+1}-x_{n}\|^2 -\frac{1}{4}\|x_{n+1}-x_{n}\|^2.
	\end{split}
	\end{equation}
    Combining this last inequality with \eqref{15c2}, yields
        \begin{equation}
    \begin{array}{l}
    a_{n+1}-(1+\alpha_n)a_n+\alpha_n a_{n-1}+\lambda_n	\beta_n f(u,x_{n+1})\\
    \leq \lambda_n\beta_n\left[\mathcal{F}_f\left(u,\frac{2p}{\beta_n}\right)-\sigma_{S_f}\left(\frac{2p}{\beta_n}\right)\right] +\alpha_n(b+1)\|x_{n}-x_{n-1}\|^2
    -(\frac{1}{2}-\frac{\alpha_n}{b})\|z_{n+1}-x_n\|^2\\
    \;\;\;-\frac{1}{4}\|x_{n+1}-x_{n}\|^2 +2\lambda_n^2\| p\|^2 .
    \end{array}
    \end{equation}	
    The proof is complete. 
\end{proof}

\begin{corollary}\label{coro-disc}
	Under hypothesis \eqref{fitz-discret}, assuming that the sequence $\left\{\alpha_n\right\}$ is nondecreasing and satisfies $\left\{\alpha_n\right\} \subseteq [0, \alpha]$ for some $\alpha \in [0, \frac{\sqrt{3}-1}{4}[$, and that $\displaystyle \sum_{n\geq 0} \lambda_n^{2} < +\infty$, we have, for each $u\in \textbf{S}$:
	\begin{itemize}
		\item [(i)] $\displaystyle \sum_{n=0}^{+\infty} \|x_{n+1}-x_n\|^2 < +\infty$;
		\item [(ii)] $\displaystyle \sum_{n=0}^{+\infty} \|z_{n+1}-x_n\|^2 < +\infty$;
		\item [(iii)] $\displaystyle \sum_{n=1}^{+\infty} \lambda_n \beta_n f(u,x_{n+1}) < +\infty$;
		\item [(iv)] $\displaystyle \lim_{n\rightarrow +\infty} \|x_n-u\|$ exists.
	\end{itemize}
\end{corollary}		 

\begin{proof}
	Let us take 
	$\delta_n = \|x_n - x_{n-1}\|^2$ in \eqref{estim1}. Since the sequence $\left\{\alpha_n\right\}$ is nondecreasing, we obtain
	\begin{equation}\label{rcle}
	\begin{split}
	&a_{n+1} - a_n - \left(\alpha_n a_n - \alpha_{n-1} a_{n-1} \right) + \lambda_n \beta_n f(u,x_{n+1}) \\
	&\leq  (b+1) \left(\alpha_n \delta_n - \alpha_{n+1} \delta_{n+1}\right) + \left(\frac{\alpha_n}{b} - \frac{1}{2} \right) \|z_{n+1} - x_n\|^2 \\
	&\quad + \left(\alpha_{n+1} (b+1) - \frac{1}{4}\right) \delta_{n+1} + \lambda_n^2 \| p\|^2 + \lambda_n \beta_n \left[\mathcal{F}_f\left(u, \frac{2p}{\beta_n}\right) - \sigma_{\textbf{S}_f} \left(\frac{2p}{\beta_n}\right)\right].
	\end{split}
	\end{equation}

	For \( n > 0 \), we have \( 0< \alpha_n \leq \alpha < \frac{\sqrt{3}-1}{4} \). Choose \( b \in ]2\alpha, \frac{1}{4\alpha}-1[ \) such that  
	\[
	\left(\frac{\alpha_n}{b} - \frac{1}{2}\right) \leq \left(\frac{\alpha}{b} - \frac{1}{2}\right) < 0
	\quad\text{and}\quad 
	\left(\alpha_{n+1}(b+1) - \frac{1}{4}\right) \leq \left(\alpha (b+1) - \frac{1}{4}\right) < 0.
	\]
	Rewriting \eqref{rcle}, we get
	\begin{equation*}
	a_{n+1} - a_n \leq \left(\alpha_n a_n - \alpha_{n-1} a_{n-1}\right) - \Delta_n + w_n,
	\end{equation*}
	where
	\[
	\Delta_n = \lambda_n \beta_n f(u,x_{n+1}) - \left(\alpha (b+1) - \frac{1}{4} \right) \delta_{n+1} - \left(\frac{\alpha}{b} - \frac{1}{2} \right) \|z_{n+1} - x_n\|^2,
	\]
	and
	\[
	w_n = (b+1) \left(\alpha_n \delta_n - \alpha_{n+1} \delta_{n+1} \right) + \lambda_n^2 \| p\|^2 + \lambda_n \beta_n \left[\mathcal{F}_f\left(u, \frac{2p}{\beta_n}\right) - \sigma_{\textbf{S}_f} \left(\frac{2p}{\beta_n}\right)\right].
	\]

	Since \( \Delta_n \geq 0 \), we analyze the summability of \( w_n \):
	\begin{itemize}
		\item $\displaystyle \sum_{n=1}^{+\infty} \left(\alpha_n \delta_n - \alpha_{n+1} \delta_{n+1} \right) \leq \alpha_1 \delta_1 = \alpha_1 \|x_1 - x_0\|^2 < +\infty$;
		\item $\displaystyle \sum_{n=1}^{+\infty} \lambda_n^2 \| p\|^2 = \| p\|^2 \sum_{n=1}^{+\infty} \lambda_n^2 < +\infty$;
		\item By hypothesis \eqref{fitz-discret}, we have
		\[
		\sum_{n=1}^{+\infty} \lambda_n \beta_n \left[\mathcal{F}_f\left(u, \frac{2p}{\beta_n}\right) - \sigma_{\textbf{S}_f} \left(\frac{2p}{\beta_n}\right)\right] < +\infty.
		\]
	\end{itemize}
	Thus, we conclude that
	\[
	\sum_{n=1}^{+\infty} w_n < +\infty.
	\]
	Since \( \alpha < 1 \), applying Lemma \ref{lem-sum1}, we obtain
	\[
	\sum_{n=0}^{+\infty} \Delta_n < +\infty
	\quad\text{and}\quad
	\lim_{n\rightarrow +\infty} a_n \quad\text{exists}.
	\]
	This completes the proof.
\end{proof}

In order to proceed with the convergence analysis, we need to choose the sequences $\{\lambda_n\}$ and $\{\beta_n\}$ such that {\(\left\{\lambda_n \right\} \in \ell^2 \setminus \ell^1\) and \(\displaystyle \liminf_{n\rightarrow +\infty} \lambda_n \beta_n > 0\)}. 

We are now ready to state and prove the first main result of this section. 

\begin{theorem}\label{thm-weak}
	Suppose that \( f \) and \( g \) are monotone and upper hemicontinuous bifunctions. Let $\lbrace x_n \rbrace$ be the sequence generated by \eqref{algo}. Under hypothesis \eqref{fitz-discret}, assume that:
	\begin{itemize}
		\item the sequence $\left\{\alpha_n\right\}$ is nondecreasing ;
		\item $\left\{\alpha_n\right\} \subseteq [0, \alpha]$ for some $\alpha \in [0, \frac{\sqrt{3}-1}{4}{[}$;
		\item \(\left\{\lambda_n \right\} \in \ell^2 \setminus \ell^1\) and $\displaystyle \liminf_{n\rightarrow +\infty} \lambda_n \beta_n > 0$.
	\end{itemize}
	Then, the sequence $\lbrace x_n \rbrace$ weakly converges to some $\bar{x} \in \textbf{S}$.
\end{theorem}

\begin{proof}
	It suffices to prove that conditions (i) and (ii) in Lemma \ref{disc-opial} are satisfied for $C = \textbf{S}$. By (iv) of Corollary \ref{coro-disc}, we know that for any $u\in \textbf{S}$, $\displaystyle \lim_{n\rightarrow +\infty} \|x_n - u\|$ exists, ensuring that condition (i) holds. 

	Next, we show that every weak cluster point $\bar{x}$ of the sequence $\lbrace x_n \rbrace$ lies in $\textbf{S}$.  
	Let $n_k \to +\infty$ as $k \to +\infty$ such that $x_{n_k} \rightharpoonup \bar{x}$. We aim to show that $\bar{x} \in \textbf{S}$.

	By the monotonicity of \( f \), inequality \eqref{algo} ensures that for all \( y \in \textbf{K} \) and for sufficiently large \( k \),
	\begin{equation}\label{13}
	\begin{array}{lll} 
	f(y, x_{{n_k}}) &\leq  &\frac{1}{\lambda_{n_k-1} \beta_{n_k-1}} \langle x_{{n_k}} - z_{n_k}, y - x_{{n_k}} \rangle\\
	&\leq& \dfrac{1}{\lambda_{n_k-1} \beta_{n_k-1}} \|x_{n_k} - z_{n_k} \| \cdot \| y - x_{n_k} \|\\
	&\leq& \dfrac{1}{\lambda_{n_k-1} \beta_{n_k-1}} \left(\|x_{n_k} - x_{n_k-1} \| + \| x_{n_k} - z_{n_k} \|\right) \cdot \| y - x_{n_k} \|.
	\end{array}
	\end{equation}

	By Corollary \ref{coro-disc}, we have \( \displaystyle \lim_{n\to +\infty} \|x_{n_k} - x_{n_k-1} \| = \lim_{n\to +\infty} \|x_{n_k-1} - z_{n_k} \| = 0 \). 
	Since \( \{x_{n_k}\} \) is bounded and \( \displaystyle \liminf_{k\rightarrow +\infty} \lambda_{n_k-1} \beta_{n_k-1} \geq \liminf_{n\rightarrow +\infty} \lambda_n \beta_n > 0 \), the weak lower semicontinuity of \( f(y, \cdot) \) ensures that \( f(y, \bar{x}) \leq 0 \) for all \( y \in \textbf{K} \). 
	Thus, Lemma \ref{lem-Minty} implies that \( \bar{x} \in {\textbf{S}_f} \).

	\medskip
	\noindent Using the first part of \eqref{algo} and the monotonicity of \( g \), for every \( u\in \textbf{S}_f \), we obtain
	\begin{equation*}
	\langle z_{n+1} - y_n, u - z_{n+1} \rangle \geq \lambda_n g(u, z_{n+1}).
	\end{equation*}
	Hence, 
	\begin{equation*}
	\|y_n - u\|^2 - \|z_{n+1} - y_n\|^2 - \|z_{n+1} - u\|^2 \geq 2\lambda_n g(u, z_{n+1}).
	\end{equation*}
	By applying inequalities \eqref{8a} and \eqref{8b}, along with the fact that
	\[
	\|x_{n+1} - u\| \leq \|z_{n+1} - u\|,
	\]
	and setting \( a_n(u) = \|x_n - u\|^2 \), we obtain, for some \( b > 0 \),
	\begin{equation*}
	2\lambda_n g(u, z_{n+1}) \leq (a_n - a_{n+1}) + \alpha_n (a_n - a_{n-1}) + \left(1 - \frac{\alpha_n}{b}\right) \|z_{n+1} - x_n\|^2 + \alpha_n (1+b) \|x_n - x_{n-1}\|^2.
	\end{equation*}
	Since \( \left\{\alpha_n\right\} \) is nondecreasing and \( \alpha_n < \alpha \), we get
	\begin{equation*}
	2\lambda_n g(u, z_{n+1}) \leq (a_n - a_{n+1}) + \left(\alpha_n a_n - \alpha_{n-1} a_{n-1}\right) + \|z_{n+1} - x_n\|^2 + \alpha(1+b) \|x_n - x_{n-1}\|^2.
	\end{equation*}
	Fixing \( N > 1 \), summing this inequality from \( n = 1 \) to \( n = N \), and letting \( N \to +\infty \), we obtain
	\begin{equation*}
	2\sum_{n \geq 1} \lambda_n g(u, z_{n+1}) \leq \|x_1 - u\|^2 + \alpha \lim _{N \to +\infty} \|x_N - u\|^2 + \sum_{n \geq 1} \|z_{n+1} - x_n\|^2 + \alpha(1+b) \sum_{n \geq 1} \|x_n - x_{n-1}\|^2.
	\end{equation*}

	By Corollary \ref{coro-disc}, we have \( \sum_{n \geq 1} \|z_{n+1} - x_n\|^2 < +\infty \) and \( \sum_{n \geq 1} \|x_n - x_{n-1}\|^2 < +\infty \). 
	Since \( \displaystyle \lim _{N \to +\infty} \|x_N - u\| \) exists, it follows that
	\[
	\sum_{n \geq 1} \lambda_n g\left(u, z_{n+1}\right) < +\infty.
	\]

	\medskip
	\noindent Since \( \sum_{n \geq 1} \lambda_n = +\infty \), applying Lemma \ref{lem-cont} gives
	\[
	\liminf _{n \to +\infty} g\left(u, x_{n+1}\right) \leq 0.
	\]
	Since \( \displaystyle \lim_{n\to +\infty} \|z_{n+1} - x_n\| = 0 \), we conclude that \( \bar{x} \) is also a weak cluster point of \( \{z_n\} \). 
	By the lower semicontinuity of \( g(u, \cdot) \), we obtain \( g(u, \bar{x}) \leq 0 \), and Lemma \ref{lem-Minty} ensures that
	\[
	g(\bar{x}, u) \geq 0, \quad \forall u \in \textbf{S}_f.
	\]
	This completes the proof.
\end{proof}

\subsection{Strong Convergence}\label{Sub3}

In this section, we assume that \( g \) is \textbf{strongly monotone} and show that the sequence \( \{x_n\} \), defined by the iteration \eqref{algo}, strongly converges to the unique solution \( u \) of \( \mathbf{(BEP}) \), without requiring the geometric assumption \eqref{fitz-discret}.

\begin{theorem}\label{strong1}
	Suppose that the bifunctions \( f \) and \( g \) are monotone and upper hemicontinuous. Additionally, assume that:
	\begin{itemize}
		\item \( g \) is \( \rho \)-strongly monotone;
		\item the sequence \( \left\{\alpha_n\right\} \) is nondecreasing;
		\item \( \left\{\alpha_n\right\} \subseteq [0, \alpha] \) for some \( \alpha \in [0, \frac{\sqrt{3}-1}{4}{[} \), and
		\begin{center}
			\( \displaystyle\sum_{n\geq0} \lambda_n = +\infty \), \( \displaystyle \lim_{n\rightarrow +\infty} \lambda_n = 0 \), {\( \displaystyle \lim_{n\rightarrow +\infty} \beta_n = +\infty \)} and \, \( \displaystyle \liminf_{n\rightarrow +\infty} \lambda_n \beta_n > 0. \)
		\end{center}
	\end{itemize}
	Then, the sequence \( \{x_n\} \) generated by \eqref{algo} strongly converges to the unique solution \( u \) of \( \mathbf{(BEP}) \).
\end{theorem}

\begin{proof}
	Uniqueness of the solution for \( \mathbf{(BEP}) \) follows from the strong monotonicity of \( g \). For existence, see \cite[Theorem 4.3]{CCR}.
	
	We now prove the convergence of the sequence \( \{x_n\} \) to the unique solution of \( \mathbf{(BEP}) \). The proof is divided into two steps. \vskip 2mm
	
	\noindent
	\textbf{Step 1: Boundedness of the sequence \( \{x_n\} \).} \vskip 2mm
	
	On one hand, taking \( y = u \) (the unique solution of \( \mathbf{(BEP}) \) in the second inequality of \eqref{psm} gives
	\[
	\lambda_n \beta_n f(x_{n+1},u) + \langle x_{n+1} - z_{n+1}, u - x_{n+1} \rangle \geq 0,
	\]
	and hence,
	\[
	\lambda_n \beta_n f(x_{n+1},u) - \|x_{n+1} - u\|^2 + \langle u - z_{n+1}, u - x_{n+1} \rangle \geq 0.
	\]
	Thus, we obtain
	\begin{equation}
	\lambda_n \beta_n f(x_{n+1},u) + \|z_{n+1} - u\|^2 - \|z_{n+1} - x_{n+1}\|^2 - \|x_{n+1} - u\|^2 \geq 0.
	\end{equation}  
	
	On the other hand, the first inequality in \eqref{psm} leads to
	\begin{equation} 
	-2\lambda_n g(z_{n+1},u) \leq \|y_n - u\|^2 - \|y_n - z_{n+1}\|^2 - \|z_{n+1} - u\|^2.
	\end{equation}
	
	Setting \( a_n(u) = \|x_n - u\|^2 \) and \( \delta_n = \|x_n - x_{n-1}\|^2 \), summing up the last two inequalities and using \eqref{8a}, \eqref{zn}, and \eqref{8b}, we obtain, for some \( b > 0 \),
	\begin{equation*}\label{cc}
	\begin{split}
	-2\lambda_n g(z_{n+1},u) &\leq  \|y_n - u\|^2 - a_{n+1}(u) - \|y_n - z_{n+1}\|^2 - \|z_{n+1} - x_{n+1}\|^2 \\
	&\leq  (a_n(u) - a_{n+1}(u)) + \alpha_n (a_n(u) - a_{n-1}(u)) + \alpha_n(1+b) \delta_n - \frac{1}{4} \delta_{n+1} \\
	&\quad + \left(\frac{\alpha_n}{b} - \frac{1}{2} \right) \|z_{n+1} - x_n\|^2 - \frac{1}{2} \|z_{n+1} - x_{n+1}\|^2 + \lambda_n \beta_n f(x_{n+1},u).
	\end{split}
	\end{equation*}
	
	Since the sequence \( \{\alpha_n\} \) is nondecreasing, for each \( n \geq 0 \), we get 
	\begin{equation}\label{see}
	\begin{split}
	&a_{n+1}(u) - \alpha_n a_n(u) + (1+b) \alpha_{n+1} \delta_{n+1} \\	
	&\leq 2\lambda_n g(z_{n+1},u) + a_n(u) - \alpha_{n-1} a_{n-1}(u) + (1+b) \alpha_n \delta_n \\
	&\quad + \left( (1+b) \alpha_{n+1} - \frac{1}{4} \right) \delta_{n+1} + \left(\frac{\alpha_n}{b} - \frac{1}{2} \right) \|z_{n+1} - x_n\|^2 + \lambda_n \beta_n f(x_{n+1},u).
	\end{split}
	\end{equation}

	As in the proof of Corollary \ref{coro-disc}, for \( n > 0 \), we have \( 0 < \alpha_n \leq \alpha < \frac{\sqrt{3}-1}{4} \). So we can find \( b \in ]2\alpha, \frac{1}{4\alpha}-1[ \) such that  
	\[
	\left(\frac{\alpha_n}{b} - \frac{1}{2}\right) \leq \left(\frac{\alpha}{b} - \frac{1}{2}\right) < 0 \quad\text{and}\quad \left(\alpha_{n+1} (b+1) - \frac{1}{4} \right) \leq \left(\alpha (b+1) - \frac{1}{4} \right) < 0.
	\]
	Thus, using inequality \eqref{see}, we obtain
	\begin{equation*}
	\begin{split}
	&a_{n+1}(u) - \alpha_n a_n(u) + (1+b) \alpha_{n+1} \delta_{n+1} \\	
	&\leq 2\lambda_n g(z_{n+1},u) + a_n(u) - \alpha_{n-1} a_{n-1}(u) + (1+b) \alpha_n \delta_n \\
	&\quad + \left(\alpha (b+1) - \frac{1}{4} \right) \delta_{n+1} + \lambda_n \beta_n f(x_{n+1},u).
	\end{split}
	\end{equation*}

	Setting \( b_n(u) = a_n(u) - \alpha_n a_{n-1}(u) + (1+b) \alpha_n \delta_n \), and noting that \( \alpha_n \leq \alpha \), we obtain, for \( n \geq 1 \),
	\begin{equation}\label{keyace}
	b_{n+1} (u) \leq 2\lambda_n g(z_{n+1},u) + b_n(u) + \left(\alpha (b+1) - \frac{1}{4} \right) \delta_{n+1} + \lambda_n \beta_n f(x_{n+1},u).
	\end{equation}

	Since \( u \in \textbf{S}_f \), we have
	\begin{equation}\label{keyac}
	b_{n+1} (u) \leq 2\lambda_n g(z_{n+1},u) + b_n(u) + \left(\alpha (b+1) - \frac{1}{4} \right) \delta_{n+1}.
	\end{equation}
	{\begin{itemize}
    \item If there exists \( n_0 \in \mathbb{N} \) such that \( \{b_{n} (u)\} \) is decreasing for all \( n \geq n_0 \), then \( b_{n+1} (u) \leq b_{n_0} (u) \). This implies that  
    \begin{align*}
        a_{n+1}(u) &\leq \alpha_{n+1} a_n(u) + b_{n_0}(u)\\
        &\leq \alpha a_n(u) + b_{n_0}(u), \quad \text{for all } n \geq n_0.
    \end{align*}
    By induction, we deduce that for all \( n \geq n_0 \geq 1 \),  
    \[
    a_{n+1}(u) \leq \alpha^{n-n_0} a_{n_0}(u) + b_{n_0}(u) \frac{1-\alpha^{n-n_0}}{1-\alpha}.
    \]
    Hence, the boundedness of the sequence \( \{a_n (u)\} \) follows.
\item Otherwise, there exists an increasing sequence \( \{k_n \} \) such that for every \( n \geq 0 \), \( b_{k_{n +1}} (u) > b_{k _n }(u) \). By Lemma \ref{lem3a}, there exists a nondecreasing sequence \( \{\sigma_n \} \) and \( n_0 > 0 \) such that  
    \[
    \lim_{n\rightarrow +\infty} \sigma_n = \infty,
    \]
    and for all \( n \geq n_0 \),  
    \[
    b_{\sigma_n}(u) < b_{\sigma_{n}+1}(u) \quad \text{and} \quad b_{n}(u) \leq b_{\sigma_{n}+1}(u).
    \]
    Taking \( n=\sigma_n \) in \eqref{keyac}, we get  
    \begin{equation}\label{14a}
        0 < b_{\sigma_{n}+1}(u) - b_{\sigma_{n}}(u) \leq \left(\alpha(b+1)-\frac{1}{4}\right)\delta_{\sigma_{n}+1} +2\lambda_{\sigma_{n}} g(z_{\sigma_{n}+1},u).
    \end{equation}
    Since \( \partial g_y(y) \neq \emptyset \), pick \( x^*(y) \in \mathcal{H} \) such that for every \( z\in \mathbf{K} \),  
    \[
    g(y,z) \geq \langle x^*(y),z-y\rangle \geq -\| x^*(y) \| \cdot \| y-z \|.
    \]
    Thus, there exists \( \gamma(y) := \| x^*(y) \| > 0 \) such that for every \( z\in \mathbf{K} \),  
    \begin{equation}\label{14}
        -g(y,z) \leq \gamma(y) \cdot \| y-z \|.
    \end{equation}
    Using the \( \rho \)-strong monotonicity of \( g \), along with \eqref{14a} and \eqref{14}, we deduce that for \( n\geq n_0 \),  
    \begin{eqnarray}
        -2\lambda_{\sigma_{n}} \gamma(u) \| z_{\sigma_{n}+1} - u \| &\leq& 2\lambda_{\sigma_{n}} g(u, z_{\sigma_{n}+1}) \\
        &\leq& \left(\alpha(b+1)-\frac{1}{4}\right)\delta_{\sigma_{n}+1} -2\lambda_{\sigma_{n}}\rho \| z_{\sigma_{n}+1} -u \|^2. \nonumber
    \end{eqnarray}
\end{itemize}
Since \( \left(\alpha(b+1)-\frac{1}{4}\right) < 0 \), we conclude that for \( n\geq n_0 \),  
\begin{equation}\label{expa}
    \| z_{\sigma_{n}+1} -u \| \leq \frac{\gamma(u)}{\rho} \quad \text{and} \quad \delta_{\sigma_{n}+1} \leq \frac{2\gamma^2(u)\lambda_{\sigma_n}}{\rho\left(\frac{1}{4}-\alpha(b+1)\right)}.
\end{equation}
By definition, we have  
\[
x_{\sigma_{n}+1} = J_{\lambda_{\sigma_{n}+1}}^{\beta_{\sigma_{n}+1}f} (z_{\sigma_{n}+1}).
\]
Since \( u \) is an equilibrium point of \( f \), it is also an equilibrium point of \( \beta_{\sigma_{n}+1} f \), which means that  
\[
u = J_{\lambda_{\sigma_{n}+1}}^{\beta_{\sigma_{n}+1}f}(u).
\]
Furthermore, since the resolvent \( J_{\lambda_{\sigma_{n}+1}}^{\beta_{\sigma_{n}+1}f} \) is firmly non-expansive, we have  
\begin{equation}\label{exp1}
    \| x_{\sigma_{n}+1} - u \| = \| J_{\lambda_{\sigma_{n}+1}}^{\beta_{\sigma_{n}+1}f} (z_{\sigma_{n}+1}) - J_{\lambda_{\sigma_{n}+1}}^{\beta_{\sigma_{n}+1}f}(u) \| \leq \| z_{\sigma_{n}+1} - u \|.
\end{equation}
Combining this inequality with \eqref{expa}, we conclude that  
\begin{equation}
    a_{\sigma_{n}+1}(u) \leq \left( \dfrac{\gamma(u)}{\rho} \right)^{2}  
    \quad \text{and} \quad  
    \delta_{\sigma_{n}+1} \leq \dfrac{2\gamma^{2}(u)\lambda_{\sigma_n}}{\rho \left(\frac{1}{4}-\alpha(b+1)\right)}.
\end{equation}
\noindent Hence, the sequence \( \{a_{\sigma_{n}+1}(u)\} \) is bounded. Since \( \{\lambda_{\sigma_{n}}\} \) is also bounded, it follows that \( \{\delta_{\sigma_{n}+1}\} \) is bounded, implying that \( \{b_{\sigma_{n}+1}(u)\} \) is bounded as well.  
Thus, there exists \( M > 0 \) such that for all \( n \geq n_0 \), we have  
\begin{equation*}
\begin{array}{lll}
    a_n(u) & \leq & \alpha_n a_{n-1}(u) + b_n(u) \\  
    & \leq & \alpha a_{n-1}(u) + b_n(u) \\  
    & \leq & \alpha a_{n-1}(u) + b_{\sigma_{n}+1}(u) \\  
    & \leq & \alpha a_{n-1}(u) + M \\  
    & \leq & \alpha^{n-n_0} a_{n_0}(u) + M {\dfrac{1 - \alpha^{n-n_0}}{1 - \alpha}}.
\end{array}
\end{equation*}
}
Therefore, the sequence \( \{a_n(u)\} \) is bounded, ensuring the boundedness of \( \{x_{n}\} \).
\vskip 2mm
	\textbf{Step 2: Strong convergence of \( \{x_n\} \) to the unique solution \( \bar{x} \) of \(\mathbf{(BEP})\).} \vskip 2mm  

Let us consider two cases:  

\vskip 2mm  
\underline{Case 1:} There exists \( n_0 \) such that the sequence \( \{b_n(\bar{x})\} \) defined by  
\[
b_n(\bar{x}) := a_n(\bar{x}) - \alpha_n a_{n-1}(\bar{x}) + (1+b)\alpha_n\delta_n
\]
is decreasing for \( n \geq n_0 \).  

Then, the limit of \( \{b_n(\bar{x})\} \) exists, and  
\[
\lim_{n\rightarrow +\infty} (b_n (\bar{x}) - b_{n+1}(\bar{x})) = 0.
\]
For \( n \geq n_0 \), we have  
\[
a_{n+1}(\bar{x}) - \alpha_{n+1} a_{n}(\bar{x}) + (1+b)\alpha_{n+1}\delta_{n+1} \leq a_n(\bar{x}) - \alpha_n a_{n-1}(\bar{x}) + (1+b)\alpha_n\delta_n,
\]
which implies that  
\[
a_{n+1}(\bar{x}) \leq a_{n}(\bar{x}) + \left[\alpha_{n+1}a_{n}(\bar{x}) - \alpha_{n}a_{n-1}(\bar{x})\right] + (1+b)\left[ \alpha_n\delta_{n} - \alpha_{n+1} \delta_{n+1} \right].
\]
Since  
\[
\sum_{n=n_0}^{+\infty} \left[\alpha_{n+1}a_{n}(\bar{x}) - \alpha_{n}a_{n-1}(\bar{x})\right] \leq \sup_{n>n_0} \alpha a_{n}(\bar{x}),
\]
and due to the fact that \( a_{n}(\bar{x}) \) is bounded (by Step 1), it follows that  
\[
\sum_{n=n_0}^{+\infty} \left[\alpha_{n+1}a_{n}(\bar{x}) - \alpha_{n}a_{n-1}(\bar{x})\right] < +\infty.
\]
By Lemma \ref{lem-sum}, and since  
\[
\sum_{n=n_0}^{+\infty} \left(\left[\alpha_{n+1}a_{n}(\bar{x}) - \alpha_{n}a_{n-1}(\bar{x})\right] + 2\left[ \alpha_n\delta_{n} - \alpha_{n+1} \delta_{n+1}\right]\right) < +\infty,
\]
the limit of \( \{a_n(\bar{x})\} \) exists. Therefore, it suffices to show that  
\[
\liminf_{n\rightarrow \infty} a_n(\bar{x}) = 0.
\]
Using \eqref{exp1}, we obtain  
\[
\lim_{n\rightarrow +\infty} \| x_{n+1} -\bar{x}\| = \liminf_{n\rightarrow +\infty}\|x_{n+1}-\bar{x}\|^{2} \leq \liminf_{n\rightarrow +\infty} \| z_{n+1} -\bar{x}\|.
\]
Since \( g \) is \( \rho \)-strongly monotone, we have  
\begin{equation}\label{k}
\begin{array}{lll}
\displaystyle \lim_{n\rightarrow +\infty} a_{n+1}(\bar{x}) & \leq & \displaystyle \liminf_{n\rightarrow +\infty}\|z_{n+1}-\bar{x}\|^{2} \\  
& \leq & \frac{1}{\rho} \displaystyle\liminf_{n\rightarrow +\infty}\left(-g(z_{n+1},\bar{x})\right) + \frac{1}{\rho} \displaystyle \limsup_{n\rightarrow +\infty}\left(-g(\bar{x},z_{n+1})\right).
\end{array}
\end{equation}
Thus, it suffices to prove that  
\[
\liminf_{n\rightarrow +\infty} \left(-g(z_{n+1},\bar{x})\right) \leq 0
\]
and  
\[
\liminf_{n\rightarrow +\infty} g(\bar{x},z_{n+1}) \geq 0.
\]
Since \( \bar{x} \in \mathbf{S}_f \), we derive from \eqref{keyac} that  
\[
b_{n+1} (\bar{x}) \leq 2\lambda_n g(z_{n+1},\bar{x}) + b_n + \left(\alpha(b+1) - \frac{1}{4}\right)\delta_{n+1}.
\]
Since \( \left(\alpha(b+1) - \frac{1}{4}\right) < 0 \), it follows that  
\begin{equation}\label{20a}
-\lambda_n g(z_{n+1},\bar{x}) \leq \frac{1}{2} \left( b_n( \bar{x}) - b_{n+1}(\bar{x}) \right).
\end{equation}
Summing inequality \eqref{20a} from \( 1 \) to \( +\infty \), we deduce that  
\[
\sum_{n=1}^{+\infty} -\lambda_n g(z_{n+1},\bar{x}) \leq \frac{1}{2} \left( b_1(\bar{x}) - \lim_{n\rightarrow +\infty} b_n(\bar{x}) \right) < +\infty.
\]
Since \( \sum_{n=0}^{+\infty} \lambda_n = +\infty \), we conclude that  
\[
\liminf_{n\rightarrow \infty} \left(-g(z_{n+1},\bar{x})\right) \leq 0.
\]
\\
Let us prove that  
\[
\liminf_{n\rightarrow +\infty} g(\bar{x}, z_{n+1}) \geq 0.
\]
Since the sequence \( \{x_n\} \) is bounded, there exists a subsequence \( \{x_{n_k}\} \) that converges weakly to some \( x \in \mathbf{K} \). 
 {From Corollary \ref{coro-disc} (ii),  the subsequence \( \{z_{n_k+1}\} \) also converges weakly to  \( x  \)}. 
Using the weak lower semicontinuity of \( g (\bar{x}, \cdot) \), we obtain  
\[
g(\bar{x}, x) \leq  \liminf_{k\rightarrow +\infty} g (\bar{x}, z_{n_k+1} ).
\]
Since \( \bar{x} \) is the unique solution of \( \mathbf{(BEP}) \), we only need to check that \( x\in \mathbf{S}_f \). To do so, using \eqref{14}, \eqref{exp1}, and \eqref{keyace}, we have for every \( y\in \mathbf{K} \),  
\begin{equation}\label{co}
f(y,x_{n+1})\leq -\frac{1}{2\lambda_n\beta_{n}}\left(b_{n+1}(y)-b_n(y)\right)+ \frac{1}{2\beta_{n}}\gamma(y)\sqrt{a_{n}(y)}.
\end{equation}
We rewrite  
\[
\begin{array}{l}
b_{n}(y)-b_{n+1}(y) \\  
= \left(a_n(y)-\alpha_{n} a_{n-1}(y)+(1+b)\alpha_{n}\delta_n\right)  
-\left(a_{n+1}(y)-\alpha_{n+1} a_{n}(y)+(1+b)\alpha_{n+1}\delta_{n+1}\right) \\  
= \left(a_n(y)- a_{n+1}(y)\right)  
+ \left(\alpha_{n+1}a_n(y)-\alpha_{n} a_{n-1}(y)\right)  
+(1+b)\left(\alpha_{n}\delta_n-\alpha_{n+1}\delta_{n+1}\right) \\  
= \left(a_{n}(\bar{x})-a_{n+1}(\bar{x})+ 2\langle x_n-x_{n+1}, \bar{x}-y\rangle\right) \\  
\quad -\left(\alpha_{n}a_{n-1}(\bar{x})-\alpha_{n+1}a_{n}(\bar{x})\right)  
+\left(\alpha_{n+1}-\alpha_n\right)\|\bar{x}-y\|^2  
+2\left(\alpha_{n+1}-\alpha_n\right) \langle x_n-\bar{x},\bar{x}-y \rangle \\  
\quad +2\alpha_n\langle x_n-x_{n-1}, \bar{x}-y\rangle  
+(1+b)\left(\alpha_{n}\delta_n-\alpha_{n+1}\delta_{n+1}\right) \\  
= b_{n}(\bar{x})-b_{n+1}(\bar{x})+2\langle x_n-x_{n+1}, \bar{x}-y\rangle  
+2\alpha_n \langle x_n-x_{n-1},\bar{x}-y \rangle \\  
\quad +\left(\alpha_{n+1}-\alpha_n\right)\|\bar{x}-y\|^2  
+2\left(\alpha_{n+1}-\alpha_n\right) \langle x_n-\bar{x},\bar{x}-y \rangle.
\end{array}
\]
Since the sequence \( \{\alpha_n\} \) is nondecreasing and bounded, we have  
\[
\lim_{n\rightarrow +\infty} \left( \alpha_{n+1}-\alpha_n \right) = 0.
\]
Further, since  
\[
\lim_{n\rightarrow +\infty}(b_n (\bar{x}) - b_{n+1}(\bar{x})) = 0 \quad \text{and} \quad \lim_{n\rightarrow +\infty} \|x_{n+1}-x_n\|= 0,
\]
it follows that  
\[
\lim_{n\rightarrow +\infty} (b_{n}(y)-b_{n+1}(y)) = 0.
\]
Using the weak lower semicontinuity of \( f (y, \cdot) \) and the facts that \( \{x_n\} \) is bounded,  
\[
\lim_{n\rightarrow +\infty} \lambda_n = 0, \quad \liminf_{n\rightarrow +\infty} \lambda_n\beta_n > 0, \quad \text{and} \quad \lim_{n\rightarrow +\infty} \beta_n = +\infty,
\]
we conclude from \eqref{co} that for every \( y\in \mathbf{K} \),  
\[
f (y, x) \leq \liminf_{n\rightarrow +\infty} f(y, x_{n+1} ) \leq 0.
\]
Hence, by Minty's lemma, we deduce that \( x\in \mathbf{S}_f \), and therefore,  
\[
0 \leq g(\bar{x}, x) \leq  \liminf_{n\rightarrow +\infty} g(\bar{x}, z_{n+1}).
\]
Thus, by \eqref{k},  
\[
\lim_{n\rightarrow +\infty} a_{n+1}(\bar{x}) \leq -\frac{1}{\rho} \liminf_{n\rightarrow +\infty} g(\bar{x},z_{n+1}) \leq 0,
\]
which implies that  
\[
\lim_{n\rightarrow +\infty} a_n(\bar{x}) = 0.
\]
\\
\underline{Case 2:} There exists a subsequence \( \{x_{n_{j}}\} \) of \( \{x_{n}\} \) such that \( b_{n_{j}}(\bar{x}) \leq b_{n_{j}+1}(\bar{x}) \) for all \( j \in \mathbb{N} \).  

By Lemma \ref{lem3a}, the sequence  
\[
\sigma(n):= \max\{k \leq n : b_k (\bar{x})< b_{k+1}(\bar{x}) \}
\]
is nondecreasing,  
\[
\lim_{n\rightarrow +\infty} \sigma(n) = \infty,
\]
and for all \( n\geq n_0 \), we have  
\[
b_{\sigma(n)}(\bar{x}) < b_{\sigma(n)+1}(\bar{x}) \quad \text{and} \quad b_{n}(\bar{x})\leq b_{\sigma(n)+1}(\bar{x}).
\]
Taking \( n=\sigma(n) \) and \( u=\bar{x} \) in \eqref{keyac}, we obtain  
\begin{equation}\label{keyad}
0 < b_{\sigma(n)+1}(\bar{x}) - b_{\sigma(n)}(\bar{x}) \leq  2\lambda_{\sigma(n)} g(z_{\sigma(n)+1},\bar{x}),
\end{equation}
which implies that \( g(z_{\sigma(n)+1},\bar{x}) \geq 0 \), and thus  
\[
\limsup_{n\rightarrow +\infty} g(z_{\sigma(n)+1},\bar{x}) \geq 0.
\]
Using again the \( \rho \)-strong monotonicity of \( g \), along with \eqref{k}, and taking the limit, we obtain  
\begin{equation}\label{kb}
\begin{array}{lll}
\displaystyle \limsup_{n\rightarrow +\infty} a_{\sigma(n)+1}(\bar{x})
& \leq & \frac{1}{\rho} \underbrace{\limsup_{n\rightarrow +\infty} -g(z_{\sigma(n)+1},\bar{x})}_{\leq 0} + \frac{1}{\rho} \displaystyle \limsup_{n\rightarrow +\infty} -g(\bar{x},z_{\sigma(n)+1})\\
&\leq& -\frac{1}{\rho} \displaystyle \liminf_{n\rightarrow +\infty} g(\bar{x},z_{\sigma(n)+1}).
\end{array}
\end{equation}
Since \( \{x_n\} \) is bounded, and similarly to Case 1, we obtain  
\[
\liminf_{n\rightarrow +\infty} g(\bar{x}, z_{\sigma(n)+1} ) \geq 0.
\]
Hence, by \eqref{k}, we conclude that  
\begin{equation}\label{k4}
\lim_{n\rightarrow +\infty} a_{\sigma(n)+1}(\bar{x}) = 0.
\end{equation}
Since \( b_n(\bar{x}) \leq b_{\sigma(n)+1} (\bar{x}) \) for each \( n \geq n_0 \), it follows that  
\[
\begin{array}{lll}
\displaystyle \lim_{n\rightarrow +\infty} a_n(\bar{x})
&\leq& \displaystyle \lim_{n\rightarrow +\infty} b_n(\bar{x})\\
&\leq& \displaystyle \lim_{n\rightarrow +\infty} b_{\sigma(n)+1} (\bar{x})\\
&\leq& \displaystyle \lim_{n\rightarrow +\infty} \left(a_{\sigma(n)}(\bar{x})+(1+b)\alpha_{\sigma(n)}\delta_{\sigma(n)}\right)\\
&\leq& \displaystyle \lim_{n\rightarrow +\infty} \left(a_{\sigma(n)}(\bar{x})+(1+b)\alpha\delta_{\sigma(n)}\right).
\end{array}
\]
Since  
\[
\delta_{\sigma(n)} = \|x_{\sigma(n)} - x_{\sigma(n)-1}\|^2 \leq 2a_{\sigma(n)}(\bar{x}) +2a_{\sigma(n)-1}(\bar{x}),
\]
we obtain  
\[
\begin{array}{lll}
\displaystyle \lim_{n\rightarrow +\infty} a_n(\bar{x})
& \leq & (1+2(1+b)\alpha) \underbrace{\displaystyle\lim_{n\rightarrow +\infty} a_{\sigma(n)}(\bar{x})}_{=0}
+ 2(1+b)\alpha \underbrace{\displaystyle \lim_{n\rightarrow +\infty} a_{\sigma(n)-1}(\bar{x})}_{=0}\\
& = & 0,
\end{array}
\]
thus guaranteeing the strong convergence of the entire sequence \( (x_n) \) to \( \bar{x} \).
\end{proof}

\subsection{Strong Convergence for Non-Strongly Monotone Bifunctions}

A bifunction \( G \) is said to be of class \( (S_+) \) if, for any sequence \( \{u_n\} \) in \( \textbf{K} \) satisfying \( u_n \rightharpoonup u \) and \( \displaystyle \limsup_{n\rightarrow+\infty} g(u, u_n) \leq 0 \), it follows that \( u_n \to u \) in norm. The notion of class \( (S_+) \) was first introduced for single-valued operators by Browder \cite{brow} in connection with the study of nonlinear eigenvalue problems and was further explored in detail in \cite{[7]}. 

In particular, given an operator \( A: \textbf{K} \to \mathcal{H} \), the bifunction \( G \) defined by 
\[
G(x, y) = \langle A x, y - x \rangle
\]
is of class \( (S_+) \) if and only if the operator \( A \) itself is of class \( (S_+) \).

\vskip 2mm

Assuming that the bifunction \( G \) is of class \( (S_+) \), the following result ensures the strong convergence of the sequence \( \{x_n\} \) generated by \eqref{algo} and the algorithm \( \textbf{(IPSA)} \) to the unique solution of \( \textbf{(BEP)} \).

\begin{theorem}\label{sconv11}
	Suppose that, in addition to the conditions of Theorem \ref{thm-weak}, \( g \) is of class \( (S_+) \). Then the whole sequence \( \{x_n\} \) generated by \( \textbf{(IPSA)} \) strongly converges to a solution \( \bar{x} \in \textbf{S} \).
\end{theorem} 

\begin{proof}   
	Following the proof of Theorem \ref{thm-weak}, the sequence \( \{z_n : n\in  \mathbb{N} \} \) weakly converges to some \( \bar{x} \in \textbf{S} \) and satisfies
	\[
	\limsup_{n\rightarrow +\infty} g(\bar{x}, z_{n+1}) \leq 0.
	\]
	Indeed, the proof of Theorem \ref{thm-weak} establishes that 
	\[
	\limsup_{n\rightarrow +\infty} \lambda_n g(\bar{x}, z_{n+1}) \leq 0.
	\]
	Since \( \displaystyle \liminf_{n\rightarrow +\infty} \lambda_n > 0 \) by assumption, we obtain the desired inequality. 
	
	Now, using the fact that \( g \) is of class \( (S_+) \), we conclude that the subsequence \( \{ z_n \} \) strongly converges to \( \bar{x} \in \textbf{S} \). Finally, since the real sequence \( \{\| x_n - z_{n+1} \|\} \) converges to zero, it follows that the whole sequence \( \{ x_n \} \) strongly converges to \( \bar{x} \).
\end{proof}

\begin{remark}
	Consider the case where \( g(x,y) = \langle A x, y - x \rangle \) for a single-valued operator \( A: \textbf{K} \to \mathcal{H} \), and recall the notion of an operator of class \( (S_+) \): if a sequence \( \{u_n\} \) in \( \textbf{K} \) satisfies \( u_n \rightharpoonup x \) and
	\[
	\limsup_{n\rightarrow +\infty} \langle A u_n, u_n - u \rangle \leq 0,
	\]
	then \( u_n \to u \) in norm.

	This condition \( (S_+) \) was first introduced by Browder \cite{brow} in connection with the study of nonlinear eigenvalue problems and was further examined in detail in \cite{[7]}.

	We have that the bifunction 
	\[
	G(x, y) = \langle A x, y - x \rangle
	\]
	is of class \( (S_+) \) if and only if the operator \( A \) is of class \( (S_+) \).
\end{remark}


\section{Numerical Example}\label{S3}

In this section, we illustrate the strong convergence result of the proposed algorithm \textbf{(IPSA)} through a numerical experiment implemented in SCILAB-6.2. \\  
Let $C=\mathbb{R}^5$, and consider the bifunction:
\[
g(x, y) = \langle A x + B y, y - x \rangle, \quad \forall x, y \in \mathbb{R}^5,
\]
where
\[
A=\begin{bmatrix}
7 & 3 & 0 & 1 & 1 \\
3 & 9 & 1 & 5 & 4 \\
0 & 1 & 10 & 3 & -4 \\
1 & 5 & 3 & 9 & -1 \\
1 & 4 & -4 & -1 & 9
\end{bmatrix}, \quad 
B=\begin{bmatrix}
5 & 3 & -1 & 1 & 2 \\
3 & 6 & 1 & 4 & 3 \\
-1 & 1 & 7 & 2 & -3 \\
1 & 4 & 2 & 7 & -2 \\
2 & 3 & -3 & -2 & 7
\end{bmatrix}.
\]
Since
\[
g(x, y) + g(y, x) = -\|x - y\|_{A-B}^2,
\]
and $A - B$ is positive definite, it follows that $g$ is strongly monotone. \\

Let $f(x, y) = \varphi(y) - \varphi(x)$ for all $x, y \in \mathbb{R}^5$, where
\[
\varphi(x) = \max \{1, \|x\|\}, \quad \forall x \in \mathbb{R}^5.
\]
It is easy to verify that $f$ is maximal monotone. 

To solve problem $\mathbf{(BEP)}$, we employ Algorithm \textbf{(IPSA)}, selecting the parameters as follows:  
starting from the initial values $x_0 = x_1 = (1,1,1,1,1)$ and setting $\lambda_n = \frac{1}{n}$, we use the iterative error $\|x_n - \bar{x}\|_2$ as a measure to evaluate the computational performance of our algorithm. 

The numerical results in Figure 1(a) illustrate the convergence rate of $\|x_n - \bar{x}\|_2$ for $\beta_n = (1+n)$ and different choices of $\alpha_n$. Meanwhile, Figure 1(b) depicts the convergence rate of $\|x_n - \bar{x}\|_2$ for various choices of $\beta_n$ when $\alpha_n = 0.1 - \frac{1}{n}$.

\begin{flushleft}
\begin{figure}[h]
    \subfloat[For $\beta_n = (1+n)$ and different $\alpha_n$]{ \includegraphics[scale=0.17]{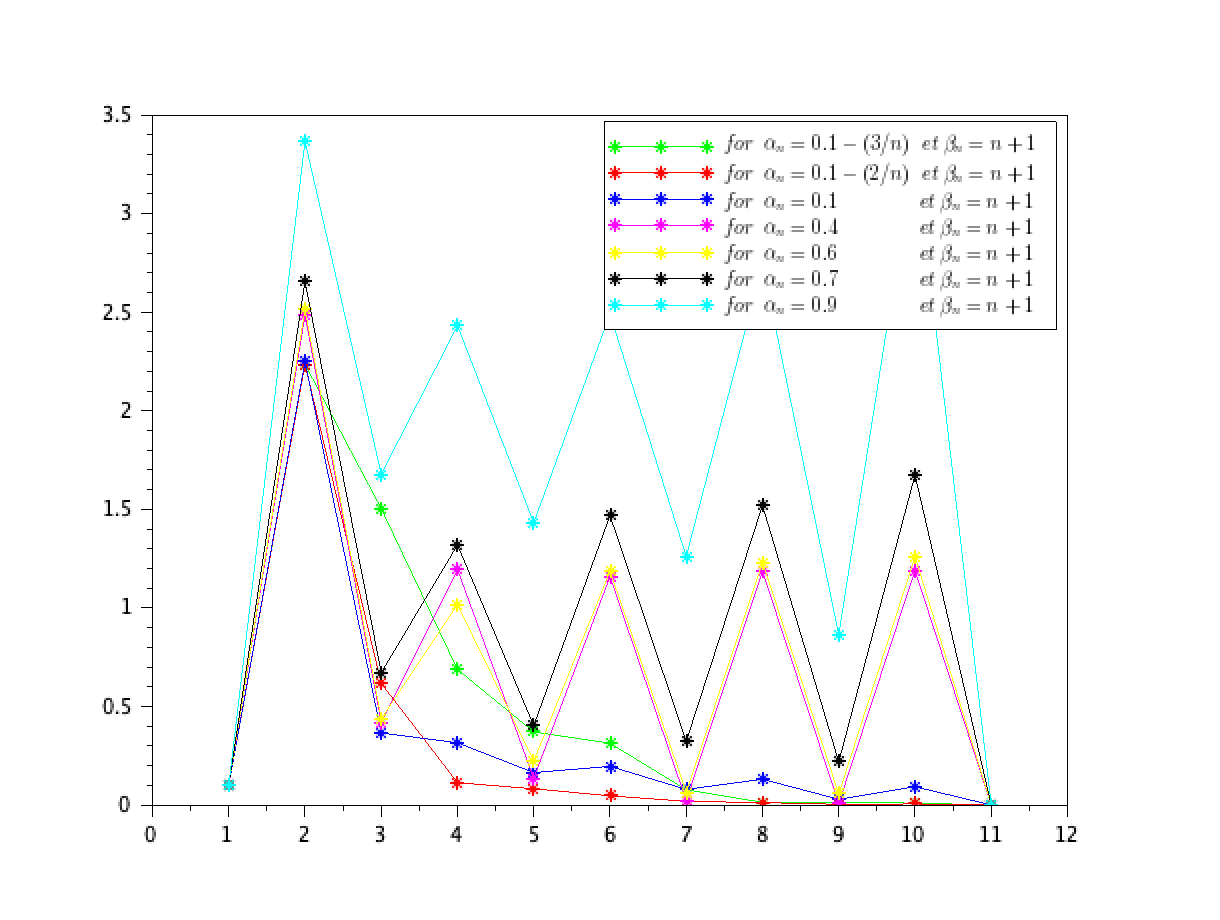}}
    \subfloat[For $\alpha_n = 0.1-\frac{1}{n}$ and different $\beta_n$]{ \includegraphics[scale=0.17]{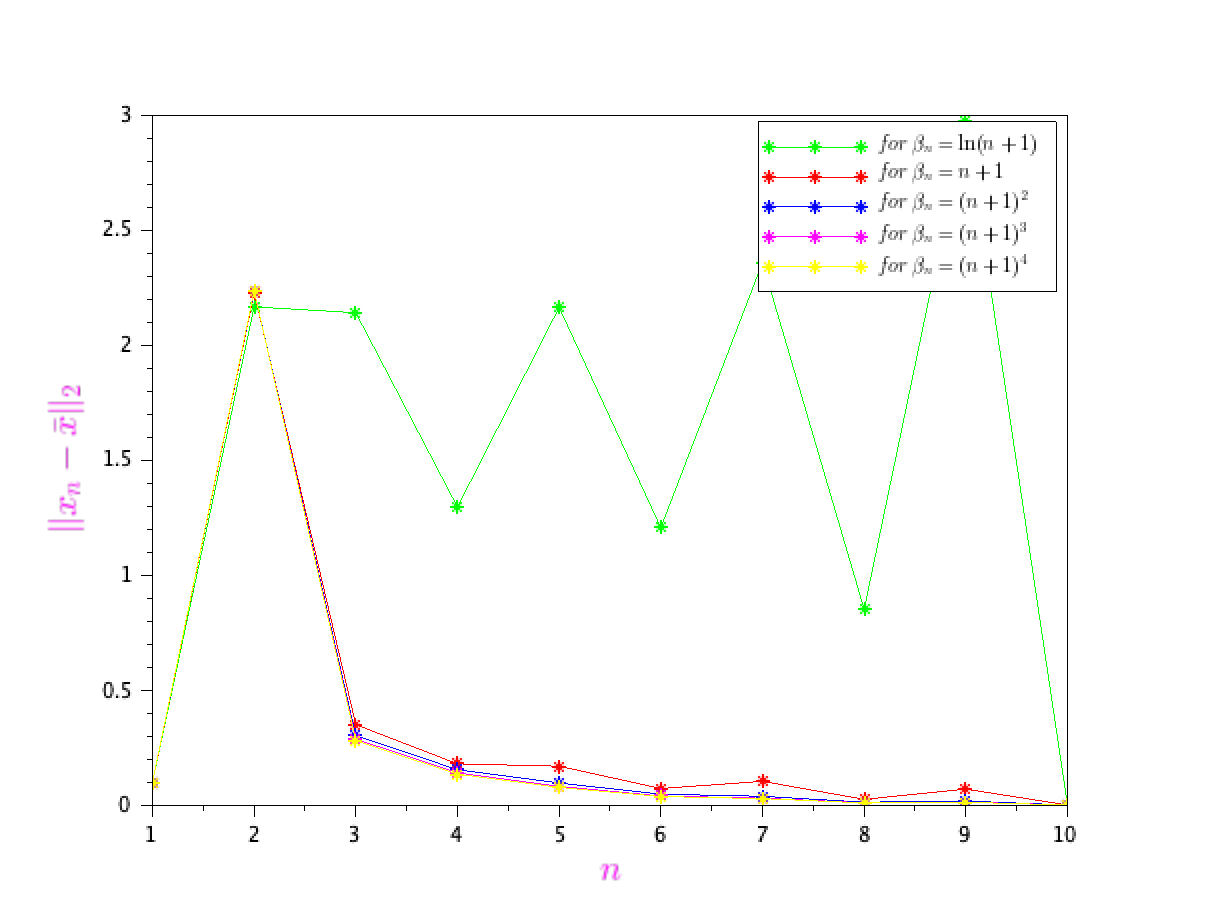}}
    \caption{The rate of convergence of $\|x_n - \bar{x}\|_2$.}
\end{figure}
\end{flushleft}

From Figure 1(b), we observe that when $\beta_n$ increases more rapidly at infinity, the sequence $\|x_n - \bar{x}\|_2$ decreases to $0$ at a higher convergence rate. Conversely, in Figure 1(a), we note that the sequence $\{\alpha_n\}$ exhibits an inverse behavior when it deviates significantly from the value $\frac{\sqrt{3}-1}{4}$. Specifically, the convergence speed of $\|x_n - \bar{x}\|_2$ deteriorates when the values of $\alpha_n$ are either too small or exceed $\frac{\sqrt{3}-1}{4}$. This highlights the importance of ensuring that $\alpha_n < \frac{\sqrt{3}-1}{4}$, in accordance with our theoretical results.

\appendix \section{Auxiliary Results}
In our convergence analysis of Algorithm \textbf{(IPSA)}, we rely on the following results.

\begin{lemma}[Discrete Opial's Lemma, \cite{opial}]\label{disc-opial}
    Let $C$ be a nonempty subset of $\mathcal{H}$ and $(x_k)_{k\geq0}$ be a sequence in $\mathcal{H}$ satisfying the following conditions:
    \begin{itemize}
        \item [(i)] For every $x\in C$,  $\displaystyle \lim_{n\rightarrow +\infty}\|x_n - x\|$ exists;
        \item [(ii)] Every weak sequential cluster point of $(x_k)_{k\geq0}$ lies in $C$.
    \end{itemize}
    Then, $(x_k)_{k\geq0}$ converges weakly to an element in $C$.
\end{lemma}

\begin{lemma}\label{lem-cont}
    Let $\{\lambda_n\}$ and $\{\mu_n\}$ be two real sequences. Suppose $\sum \lambda_n \mu_n < +\infty$ and $\sum \lambda_n = +\infty$, then 
    \[
    \liminf _{n \rightarrow+\infty} \mu_n \leq 0.
    \]
\end{lemma}
\begin{proof}
    If $\displaystyle \liminf _{n \rightarrow+\infty} \mu_n =\sup_{N>0}\inf_{n\geq N}\mu_n> 0$, then for some $N_0 > 0$ and $\delta > 0$, we have $\mu_n\geq\delta$ for each $n\geq N_0$. This leads to   
    \[
    +\infty = \delta \sum_{n\geq N_0} \lambda_n \leq \sum_{n\geq 1} \lambda_n \mu_n < +\infty,
    \]
    a contradiction.
\end{proof}	
\begin{lemma}\label{lem-sum}
	Let $0\leq p< 1$, and let $\{b_k\}$ and $\{w_k\}$ be two sequences of nonnegative numbers such that, for all $k\geq 0$, 
	\begin{equation*}
	b_{k+1}\leq pb_k+w_k.
	\end{equation*} 
	If $\{w_k\}$ is bounded, then  $\{b_k\}$ is bounded.
	Furthemore, if \;$\sum_{k=0}^{+\infty} w_k< +\infty$, then \;$\sum_{k=0}^{+\infty} b_k< +\infty$.
\end{lemma}

\begin{proof}
	If 	$\{w_k\}$ is bounded, then exists $C>0$ such that for all $k\geq 1 ,$  $w_k\leq C.$  Hence 
	$$b_{k+1}\leq pb_k+C,\quad \mbox{for all}\quad k\geq 1.$$
	Recursively we obtain for all $n\geq n_0\geq 1$
	\begin{align*}
	b_{n+1}&\leq p^{n-n_0}b_{n_0}+C(1+p+p^2+...+p^{n-n_0-1})
	\\
	&=p^{n-n_0}a_{n_0}+C\frac{1-p^{n-n_0}}{1-p}.
	\end{align*}
	Therefore the sequence $\{b_k\}$ is bounded. On the other side
	we have $$(1-p)b_k\leq b_k-b_{k+1}+w_k.$$	
	Summing up from  $k=0$ to $n$, we get
	\begin{align*}
	(1-p)\sum_{k=0}^{n}b_k&\leq \sum_{k=0}^{n}(b_k-b_{k+1})+\sum_{k=0}^{n}w_k\\
	&=b_0-b_{n+1}+\sum_{k=0}^{n}w_k \\
	&\leq b_0+\sum_{k=0}^{n}w_k.
	\end{align*}
	Since $1-p\geq 0$ and $\sum_{k=0}^{+\infty} w_k< +\infty$, we conclude that $\sum_{k=0}^{+\infty} b_k< +\infty$.
\end{proof}
\begin{lemma}\label{lem-sum1}
	Let $0\leq \alpha <1$, and let{,} $\{\alpha_k\}, $ $\{a_k\},$ $\{\Delta_k\}$ and $\{w_k\}$ be sequences of nonnegative numbers such that $\left\{\alpha_k\right\}\subseteq[0, \alpha]$ and  for all $k\geq 1$,  
	\begin{equation}\label{de}
	a_{k+1}\leq (\alpha_k+1)a_{k}- \alpha_{k-1} a_{k-1}-\Delta_k+w_k.
	\end{equation} 
	If \;$\sum_{k=0}^{+\infty} w_k< +\infty$, then
	\begin{description}
		\item[$i)$]	 $a_k$ is bounded;
		\item[$ii)$]  $\sum_{k=0}^{+\infty}\Delta_k< +\infty$;
		\item[$iii)$]  $\displaystyle \lim_{k\rightarrow +\infty}a_k$ exists.
	\end{description}
\end{lemma}

\begin{proof}
$i)$: Summing the inequality \eqref{de} from $k=1$ to $n$, we obtain   
	$$a_{n+1}-a_{1}\leq \left(\alpha_n a_{n}-\alpha_0 a_{0}\right)-\sum_{k=1}^{n}\Delta_k+\sum_{k=1}^{n}w_k.$$
	This implies 
	$$a_{n+1}\leq \alpha a_{n}+ a_{1}+ \sum_{k=1}^{n}w_k,$$
	since  $\sum_{k=0}^{+\infty} w_k< +\infty$. By Lemma \ref{lem-sum}, we conclude that 
	$\{a_n\}$  is bounded. \\
	
	\noindent
$ii)$: We have 
	$$ \sum_{k=1}^{n}\Delta_k \leq a_{1} +\alpha a_{n}+\sum_{k=1}^{n}w_k.$$
	Since $\{a_n\}$    is bounded and  $\sum_{k=0}^{+\infty} w_k< +\infty$, we conclude  
	$$\sum_{k=1}^{+\infty}\Delta_k<+\infty.$$	
$iii)$: 
	Taking the positive part in the inequality \eqref{de}, we find
	\begin{equation*}
	[a_{n+1}-a_n]_+\leq\alpha[a_n-a_{n-1}]_+ +w_k.
	\end{equation*}
	Let us apply Lemma \ref{lem-sum}   by taking $b_n=[a_n-a_{n-1}]_+,$ we conclude
	$$\sum_{n=1}^{+\infty}[a_n-a_{n-1}]_+<+\infty.$$ 
	Since $a_n$ is nonnegative, this implies the existence of $\displaystyle \lim_{n\rightarrow +\infty} a_n$, which ends the proof.		
\end{proof}

We also need the following two technical {lemmata}.
\begin{lemma}\cite{Bauchk}\label{lem5-a}
	For all $x, y \in \mathcal{H}$ and $\beta\in\mathbb{R},$ the following equality holds,
	\begin{equation*}
	\|\beta x+ (1-\beta)y\|^{2} = \beta\|x\|^{2} + (1-\beta)\|y\|^{2} -\beta(1-\beta)\|x-y\|^{2}.
	\end{equation*}
\end{lemma}

\begin{lemma}\cite{CR1}\label{lem3a}
	Let $\{a_n\}$ be a sequence of real numbers that does not decrease at
	infinity, in the sense that there exists a subsequence $\{a_{n_k}\}_ {k\geq0}$ of $\{a_n \}$ which satisfies
	$$a_{n_k} < a_{n_k+1} \quad \mbox{for all} \quad k \geq 0.$$ Then, the sequence of integers $\{\sigma(n)\}_{n\geq n_0}$ defined
	by
	$\sigma(n) := \max\{k \leq n : a_k < a_{k+1} \}$
	is a nondecreasing sequence verifying $\displaystyle \lim_{n\rightarrow +\infty} \sigma(n) = \infty$ and, for all $n \geq n_0$
	$$a_{\sigma(n)} < a_{\sigma(n)+1}\quad \mbox{and}\quad a_{n}\leq a_{\sigma(n)+1}.$$
\end{lemma}

\section*{Funding}
This research was partially supported by the FASIC project 49763ZL  and  by a public grant as part of the Investissement d'avenir project, reference ANR-11-LABX-0056-LMH, LabEx LMH.



\end{document}